\documentclass[12pt]{article}
\usepackage{a4}
\usepackage{amsthm}
\usepackage{amsmath}
\usepackage{amssymb}
\usepackage{amsfonts}
\usepackage{stmaryrd}
\usepackage{cite}
\usepackage{epsfig}
\usepackage{mathtools}
\newtheorem{conjecture}{Conjecture}
\newtheorem{theorem}{Theorem}
\newtheorem{lemma}[theorem]{Lemma}
\newtheorem{proposition}[theorem]{Proposition}

\newtheorem{corollary}[theorem]{Corollary}
\newcommand\cut[2]{d\left(#1,#2\right)_{\square}}
\newcommand\CC{{\mathbb C}}
\newcommand\JJ{{\mathbb J}}
\newcommand\NN{{\mathbb N}}
\newcommand\RR{{\mathbb R}}
\newcommand{\dd}{\;\mathrm{d}}
\newcommand\Aut{\mathrm{Aut}}
\newcommand\Real{\mathrm{Re}\;}
\newcommand\Imag{\mathrm{Im}\;}
\newcommand\Trace{\mathrm{Tr}\;}
\newcommand\Cycl{\mathrm{Cycl}\;}
\newcommand\unit{\iota}
\newcommand\wsigma{\widehat{\sigma}}

\begin{document}
\title{Cycles of a given length in tournaments\thanks{The work of the first, second and last authors has received funding from the European Research Council (ERC) under the European Union’s Horizon 2020 research and innovation programme (grant agreement No 648509).  The second and the last were also supported by the MUNI Award in Science and Humanities (MUNI/I/1677/2018) of the Grant Agency of Masaryk University. The work of the third author was supported by NSF Postdoctoral Fellowship Award DMS-1705204. This publication reflects only its authors' view; the European Research Council Executive Agency is not responsible for any use that may be made of the information it contains.}\\ {\large\emph{(Dedicated to the memory of Robin Thomas)}}}
\author{Andrzej Grzesik\thanks{Faculty of Mathematics and Computer Science, Jagiellonian University, {\L}ojasiewicza 6, 30-348 Krak\'{o}w, Poland. E-mail: {\tt Andrzej.Grzesik@uj.edu.pl}.}\and
        Daniel Kr{\'a}l'\thanks{Faculty of Informatics, Masaryk University, Botanick\'a 68A, 602 00 Brno, Czech Republic, and Mathematics Institute, DIMAP and Department of Computer Science, University of Warwick, Coventry CV4 7AL, UK. E-mail: {\tt dkral@fi.muni.cz}.}\and
        L{\'a}szl{\'o} M. Lov{\'a}sz\thanks{Department of Mathematics, Massachusetts Institute of Technology, Cambridge, MA, 02139, USA. E-mail: {\tt lmlovasz@mit.edu}.}\and
        Jan Volec\thanks{Department of Mathematics, Faculty of Nuclear Sciences and Physical Engineering, Czech Technical University in Prague, Trojanova 13, 120 00 Prague, Czech Republic. Previous affiliation: Faculty of Informatics, Masaryk University, Botanick\'a 68A, 602 00 Brno, Czech Republic. E-mail: {\tt jan@ucw.cz}.}}
	
\date{}
\maketitle
\begin{abstract}
We study the asymptotic behavior of the maximum number of directed cycles of a given length in a tournament:
let $c(\ell)$ be the limit of the ratio of the maximum number of cycles of length $\ell$ in an $n$-vertex tournament and
the expected number of cycles of length $\ell$ in the random $n$-vertex tournament,
when $n$ tends to infinity.
It is well-known that $c(3)=1$ and $c(4)=4/3$.
We show that $c(\ell)=1$ if and only if $\ell$ is not divisible by four, which settles a conjecture of Bartley and Day.
If $\ell$ is divisible by four,
we show that
$1+2\cdot\left(2/\pi\right)^{\ell}\le c(\ell)\le 1+\left(2/\pi+o(1)\right)^{\ell}$ and
determine the value $c(\ell)$ exactly for $\ell = 8$.
We also give a full description of the asymptotic structure of tournaments with the maximum number of cycles of length $\ell$
when $\ell$ is not divisible by four or $\ell\in\{4,8\}$.
\end{abstract}

\section{Introduction}
\label{sec:intro}

In this paper, we address one of the most natural extremal problems concerning tournaments:
\emph{What is the maximum number of cycles of a given length that can be contained in an $n$-vertex tournament?}
The cases of cycles of length three and four are well-understood.
An $n$-vertex tournament has at most $\frac{n(n^2-1)}{24}$ cycles of length three (cyclic triangles) if $n$ is odd, and
at most $\frac{n(n^2-4)}{24}$ if $n$ is even; both bounds are the best possible.
This result can be traced back to 1940 to the work of Kendall and Babington Smith~\cite{KenB40} and of Szele~\cite{Sze43}, also see~\cite{Moo15}, and
it is well-known that the number of cycles of length three is determined by the degree sequence of a tournament~\cite{Goo59}.
Beineke and Harary~\cite{BeiH65} and Colombo~\cite{Col64} proved in the 1960s the best possible bounds on the number of cycles of length four:
$\frac{n(n^2-1)(n-3)}{48}$ when $n$ is odd and $\frac{n(n^2-4)(n-3)}{48}$ when $n$ is even.
The asymptotics of the case of cycles of length five was determined only recently by Komarov and Mackey~\cite{KomM17}
who showed that the number of cycles of length five is asymptotically maximized if and only if the tournament is almost regular,
i.e., the behavior in this case is completely analogous to that of cycles of length three;
exact results on cycles of length five for regular tournaments and tournaments of odd order
were obtained by Savchenko in~\cite{Sav16,SavXX}.
In this paper,
we employ algebraic techniques
to provide asymptotically optimal results on the maximum number of cycles for each cycle length not divisible by four, and
determine the limit behavior in the case of cycle lengths divisible by four.

To state our results precisely, we need to fix some notation.
Let $C(n,\ell)$ be the maximum number of cycles of length $\ell$ in an $n$-vertex tournament.
We compare this quantity to the expected number of cycles of length $\ell$ in a random $n$-vertex tournament,
which is $R(n,\ell)=\frac{(\ell-1)!}{2^{\ell}}\binom{n}{\ell}$, and define
\[c(\ell)=\lim_{n\to\infty}\frac{C(n,\ell)}{R(n,\ell)}.\]
In particular, the results that we have mentioned earlier imply that $c(3)=c(5)=1$ and $c(4)=4/3$.
Bartley and Day conjectured the following.
\begin{conjecture}[{Bartley~\cite[Conjecture 104]{Bar18} and Day~\cite[Conjecture 40]{Day17}}]
\label{conj:equiv}
For $\ell\ge 3$, it holds that $c(\ell)=1$ if and only if $\ell$ is not divisible by four.
\end{conjecture}
We remark that the statement of Conjecture~\ref{conj:equiv} has been proven for \emph{regular} tournaments
by Savchenko~\cite{Sav16}, also see~\cite{Sav17}, and also jointly by Bartley and Day~\cite{Day17,Bar18}, 
and for $\ell\le 8$ by Bartley~\cite[Theorem 109]{Bar18}.

Since it is known that $c(\ell)>1$ for all $\ell$ divisible by four,
the following theorem, which is implied by Theorems~\ref{thm:ck1} and~\ref{thm:ck2}, settles the conjecture.
\begin{theorem}
\label{thm:ck1+2}
Let $\ell\ge 3$. If $\ell$ is not divisble by four, then $c(\ell)=1$.
\end{theorem}
If $\ell$ is divisible by four, we establish an asymptotically tight upper bound (Theorem~\ref{thm:ck4}):
\[1+2\cdot\left(2/\pi\right)^{\ell}\le c(\ell)\le 1+\left(2/\pi+o(1)\right)^{\ell}\]
Our asymptotic result on $c(\ell)$ for $\ell$ divisible by four
provides a strong evidence for the following conjecture on the value of $c(\ell)$ for such $\ell$ (we remark that Conjecture~\ref{conj:div4} is stated in~\cite{Day17} by giving an extremal construction,
which we mention at the end of Section~\ref{sec:prelim});
the conjecture is an extension of an earlier problem posed by Savchenko~\cite{Sav16} for regular tournaments.
\begin{conjecture}[{Bartley~\cite[Conjecture 106]{Bar18} and Day~\cite[Conjecture 45]{Day17}}]
\label{conj:div4}
If $\ell$ is divisible by four, then
\[c(\ell)=1+2\cdot\sum_{i=1}^{\infty}\left(\frac{2}{(2i-1)\pi}\right)^{\ell}.\]
\end{conjecture}
Our asymptotic result agrees with the conjecture on the dominant term of the sum.
We also show that $c(8)=332/315$ (Theorem~\ref{thm:c8}) and classify the extremal constructions (Theorem~\ref{thm:c48});
note that the value of $c(8)$ is the one given in Conjecture~\ref{conj:div4}.

In addition to the results on the value of $c(\ell)$,
we have also been able to determine the asymptotic structure of extremal tournaments when $\ell$ is not divisible by four.
If $\ell$ is odd,
then tournaments achieving the maximum number of cycles of length $\ell$ are exactly those that are almost regular (Theorem~\ref{thm:ck1}), and
if $\ell$ is even but not divisible by four,
then a tournament achieves the maximum number of cycles of length $\ell$ if and only if it is quasirandom (Theorem~\ref{thm:ck2}).
In particular, maximizing the density of cycles of length $4k+2$ is a \emph{quasirandom-forcing} property,
i.e., a property that a tournament has if and only if it is quasirandom.
We remark that in the induced setting,
in addition to the density of transitive tournaments with four or more vertices,
which is known to be quasirandom-forcing, see~\cite{CorR17} and \cite[Exercise 10.44]{Lov93},
there is only one additional tournament such that its density is quasirandom-forcing~\cite{BucLSS21,CorPS19,HanKKMPSV19},
which is the unique $5$-vertex strongly connected tournament with diameter four.

\section{Preliminaries}
\label{sec:prelim}

In this section, we fix the notation used throughout the paper and present analytic and algebraic tools needed for our arguments.
The set of the first $n$ positive integers is denoted by $[n]$.
A \emph{tournament} is an orientation of a complete graph, and
the \emph{random tournament} is an orientation of a complete graph
where each edge is directed with probability $1/2$ in each of the two possible directions independently of the other edges.
Let $C(T,\ell)$ be the ratio of the number of cycles of length $\ell$ in an $n$-vertex tournament $T$ and
the expected number of cycles of length $\ell$ in the random $n$-vertex tournament.
In particular, the maximum value of $C(T,\ell)$,
where the maximum is taken over all $n$-vertex tournaments,
is $C(n,\ell)/R(n,\ell)$.

The \emph{adjacency matrix} $A$ of a tournament $T$ is the zero-one matrix
with rows and columns indexed by vertices of $T$ such that
$A_{ij}=1$ iff $T$ contains an edge from the $i$-th vertex to the $j$-th vertex.
The \emph{tournament matrix} of an $n$-vertex tournament $T$
is the matrix obtained from the adjacency matrix of $T$ by setting its diagonal entries to be equal to $1/2$ and
then dividing each entry of the matrix by $n$.
We say that a real square matrix $A$ of order $n$
is \emph{skew-symmetric} if $A=-A^T$, i.e., $A_{ij}=-A_{ji}$ for all $i,j\in [n]$, and
$A$ is \emph{complementary} if $A$ is non-negative and $A_{ij}+A_{ji}=1/n$ for all $i,j\in [n]$.
In particular, the tournament matrix of a tournament is complementary.
Finally,
if $A$ is an $n\times n$ matrix, then its \emph{Frobenius norm}, which is denoted by $\|A\|_F$, is
\[\|A\|_F=\sqrt{\sum_{i,j\in [n]}A_{ij}^2}.\]
We recall that $\|Av\|\le\|A\|_F\cdot\|v\|$ for every vector $v\in\RR^n$.

Since the trace of the $\ell$-th power of the adjacency matrix of $T$
is the number of closed walks of length $\ell$, we obtain the following;
note that we state the next proposition for tournament matrices rather than adjacency matrices.

\begin{proposition}
\label{prop:eigen}
Let $A$ be the tournament matrix of an $n$-vertex tournament $T$,
$\lambda_1,\ldots,\lambda_n$ be its eigenvalues, and $\ell\ge 3$ an integer.
It holds that
\[C(T,\ell)=\frac{2^{\ell}}{n^\ell}\cdot\sum_{i=1}^n \lambda_i^{\ell}+O(n^{-1}).\]
\end{proposition}

We next recall some basic properties of tournament matrices, and more generally complementary matrices, used in~\cite{ChaGKN19};
we remark that similar results were also used earlier by Brauer and Gentry~\cite{BraG68}.

\begin{proposition}
\label{prop:matrix}
Let $A$ be a complementary matrix.
There is a positive real number $\rho$ such that $\rho$ is a real eigenvalue of $A$, and
the absolute value of each eigenvalue of $A$ is at most $\rho$.
In addition,
each eigenvalue of $A$ has non-negative real part, and the sum of the eigenvalues is equal to $1/2$.
\end{proposition}

\subsection{Tournament limits}
\label{subsec:limit}

We now define tournament limits,
which are analogous to graph limits described in detail in the monograph by Lov\'asz~\cite{Lov12} and
which were used earlier in~\cite{ChaGKN19,Tho18,ZhaZ20}, also see~\cite{DiaJ08,LovS10i} for related concepts.
Most of the results translate readily from the setting of graphs to that of tournaments.
However,
the results on the convergence of spectra seem to be an exception as we point out further.

A \emph{tournamenton} is a measurable function $W:[0,1]^2\to [0,1]$ such that $W(x,y)+W(y,x)=1$ for all $(x,y)\in[0,1]^2$.
The \emph{density} of a tournament $T$ in a tournament $T_0$, which is denoted by $d(T,T_0)$,
is the probability that a uniformly randomly chosen subset of $|T|$ vertices of $T_0$ induces a tournament isomorphic to $T$;
if $|T|>|T_0|$, we set $d(T,T_0)=0$.
The \emph{density} of an $n$-vertex tournament $T$ in a tournamenton $W$ is defined as
\[d(T,W)=\frac{|T|!}{|\Aut(T)|!}\int_{x_1,\ldots,x_n\in [0,1]}\prod_{i\to j}W(x_i,x_j)\dd x_1\cdots x_n,\]
where the product is taken over all $i$ and $j$ such that the $i$-th vertex is joined by an edge to the $j$-th vertex.
Two tournamentons $W$ and $W'$ are \emph{weakly isomorphic} if $d(T,W)=d(T,W')$ for every tournament $T$.

We say that a sequence $(T_n)_{n\in\NN}$ of tournaments is \emph{convergent}
if $|T_n|$ tends to infinity and the sequence $d(T,T_n)$ converges for every tournament $T$.
A tournamenton $W$ is a \emph{limit} of a convergent sequence $(T_n)_{n\in\NN}$ of tournaments
if $d(T,W)$ is equal to the limit of $d(T,T_n)$ for every tournament $T$.
For example,
the tournamenton equal to $1/2$ everywhere
is the limit of the sequence of random $n$-vertex tournaments with probability one.
A tournamenton $W$ is called \emph{regular} if
\[\int_{[0,1]}W(x,y)\dd y=1/2\]
for almost every $x\in [0,1]$;
such tournamentons are limits of tournaments
where the in-degrees and out-degrees of most of the vertices are asymptotically equal to half of the total number of vertices.

An analogous line of arguments as in the graph case yields that
every convergent sequence of tournaments has a limit and
every tournamenton is a limit of a convergent sequence of tournaments.
In particular, a tournamenton $W$ is a limit of $W$-random tournaments that we next define.
An $n$-vertex \emph{$W$-random tournament} is obtained as follows:
sample $n$ points $x_1,\ldots,x_n$ uniformly and independently in $[0,1]$ and
orient the edge between the $i$-th and $j$-th vertex from the $i$-th vertex to the $j$-th vertex with probability $W(x_i,x_j)$.
Note that the expected density of a tournament $T$ in a $W$-random tournament is equal to $d(T,W)$.

We next introduce the quantity $C(W,\ell)$, which is the limit analogue of $C(T,\ell)$ defined earlier:
\[C(W,\ell)=2^{\ell}\int_{x_1,\ldots,x_\ell\in [0,1]}W(x_1,x_2)W(x_2,x_3)\cdots W(x_{\ell-1},x_{\ell})W(x_{\ell},x_1)\dd x_1\cdots x_{\ell}.\]
It follows that $c(\ell)$ is the maximum of $C(W,\ell)$ where the maximum is taken over all tournamentons $W$ (and
it can be shown that the maximum is indeed attained).
Note that the expected value of $C(T,\ell)$ for an $n$-vertex $W$-random tournament $T$, $n\ge\ell$, is equal to $C(W,\ell)$.

If $W$ and $W'$ are two tournamentons, then the \emph{cut distance} between $W$ and $W'$,
which is denoted by $\cut{W}{W'}$, is defined as
\[\cut{W}{W'}=\sup_{X,Y\subseteq [0,1]}\left|\int_{X\times Y}W(x,y)-W'(x,y)\dd x\dd y\right|,\]
where the supremum is taken over all measurable subsets $X$ and $Y$ of $[0,1]$.
As in the graph case, it can be shown that 
\begin{equation}
\left|d(T,W)-d(T,W')\right|\le |T|^2\cut{W}{W'}\label{eq:cut}
\end{equation}
for every tournament $T$ and all tournamentons $W$ and $W'$.

Every complementary matrix $A$ of order $k$ can be associated with a tournamenton $W$ as follows:
the interval $[0,1]$ is split into $k$ disjoint measurable sets $U_1,\ldots,U_k$ each of measure $1/k$ and
$W(x,y)=k\cdot A_{ij}$ if $x\in U_i$ and $y\in U_j$;
tournamentons that can be obtained in this way are called \emph{step tournamentons}.
A \emph{step approximation} of a tournamenton $W$ is a complementary matrix $A$ such that
there exists a partition of the interval $[0,1]$ to $k$ disjoint measurable sets $U_1,\ldots,U_k$ each of measure $1/k$ that
\[A_{ij}=k\int_{U_i\times U_j}W(x,y)\dd x\dd y\]
for every $i,j\in [k]$.
The step tournamenton associated with a step approximation $A$ and the sets $U_1,\ldots,U_k$ used to define $A$ is denoted by $W[A]$.
We say that a sequence of step approximations $(A_n)_{n\in\NN}$ of a tournamenton $W$ is \emph{convergent}
if $\cut{W}{W[A_n]}$ converges to zero.
The Regularity Lemma yields that for every $\varepsilon>0$ and every tournamenton $W$,
there exists a step approximation $A$ of the tournamenton $W$ such that $\cut{W}{W[A]}\le\varepsilon$.
In particular, every tournamenton has a convergent sequence of step approximations and
we obtain using \eqref{eq:cut} the following.
\begin{proposition}
\label{prop:cut}
If $W$ is a tournamenton and $(A_n)_{n\in\NN}$ is a convergent sequence of its step approximations,
then
\[d(T,W)=\lim_{n\to\infty}d(T,W[A_n])\]
for every tournament $T$.
In particular, it holds that
\[C(W,\ell)=2^\ell\lim_{n\to\infty}\Trace A_n^\ell\]
for every $\ell\ge 3$.
\end{proposition}

We finish this subsection by describing a tournamenton that
is believed to be extremal for Conjecture~\ref{conj:div4} for every $\ell$ divisible by four.
For $x,y\in[0,1]$,
define $W_C(x,x)=1/2$, $W_C(x,y)=1$ if $y\in (x-1,x-1/2)\cup (x,x+1/2]$, and $W_C(x,y)=0$ otherwise.
The tournamenton $W_C$, which we refer to as the carousel tournamenton, is depicted in Figure~\ref{fig:carouselle}, and
it is the limit of the following tournaments described in~\cite{Day17} in relation to Conjecture~\ref{conj:div4}:
take vertices $0,\ldots,2n$ and join a vertex $i$ to the vertices $i+1,\ldots,i+n$ (computations modulo $2n+1$).
These tournaments are called \emph{carousel tournaments}.
The value of $C(W_C,\ell)$ for every positive integer $\ell$ divisible by four
is equal to the value of $c(\ell)$ given in Conjecture~\ref{conj:div4}.

\begin{figure}
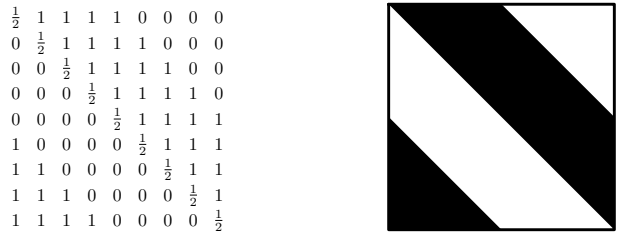

\begin{center}
\epsfysize 3cm
\epsfbox{tourn-ck-7.mps}
\hskip 2cm
\epsfbox{tourn-ck-1.mps}
\end{center}
\caption{The tournament matrix of the $9$-vertex carousel tournament and the carousel tournamenton.
         The origin of the coordinate system is in the top left corner, the $x$-axis is vertical and the $y$-axis is horizontal.
	 The black color represents the value $1$ and the white color the value $0$.}
\label{fig:carouselle}
\end{figure}

\subsection{Spectral properties of tournaments and their limits}
\label{subsec:spectrum}

A tournamenton $W$ can be viewed as a linear operator from $L_2[0,1]$ to $L_2[0,1]$
defined as
\[(Wf)(x)=\int_{[0,1]}W(x,y)f(y)\dd y.\]
Since this operator is compact (as all Hilbert-Schmidt integral operators are),
its spectrum $\sigma(W)$ is either finite or countably infinite,
the only accumulation point of $\sigma(W)$ can be zero, and
every non-zero element of $\sigma(W)$ is an eigenvalue of $W$.
Moreover, for every non-zero $\lambda\in\sigma(W)$, there exists $k_{\lambda}\in\NN$ such that
the kernels of $(W-\lambda)^{k_{\lambda}}$ and $(W-\lambda)^{k_{\lambda}+1}$ are the same and
their dimension is finite.
For example, the spectrum of the carousel tournamenton $W_C$ defined at the end of Subsection~\ref{subsec:limit}
consists of $1/2$, $\pm\unit/((2k-1)\pi)$ for $k\in\NN$, and $0$.

It is plausible that
if $(T_n)_{n\in\NN}$ is a convergent sequence of tournaments and $W$ is a limit tournamenton,
then the normalized spectra of the tournament matrices of $T_n$ converge to the spectrum of $W$
in the sense used in the graph setting in~\cite[Section 6]{BorCLSV12}, also see~\cite[Chapter 11]{Lov12}.
However, the equality between the density of cycles of length $\ell$ and
the trace of $W^\ell$ for $\ell\ge 3$ (note that $W^\ell$ is a trace-class operator for $\ell\ge 2$),
which forms the core of the argument in~\cite[Section 6]{BorCLSV12} and is straightforward in the case of graphons,
is not obvious in the case of tournamentons.
While we establish a close analogy of this equality as \eqref{eq:limitpower} in Proposition~\ref{prop:limitset},
it does not seem to be strong enough to give results completely analogous to those on the convergence of spectra of graph limits.

To proceed with our exposition, we need to define a notion of convergence for multisets of complex numbers.
If $z$ is a complex number,
we write $N_{\varepsilon}(z)$ for the set of all complex numbers $z'$ with $|z-z'|\le\varepsilon$.
We say that a sequence $(X_i)_{i\in\NN}$ of multisets of complex numbers \emph{converges as multisets}
to a multiset $X$ if the following holds:
\begin{itemize}
\item if $z$ is an element of $X$ with a finite multiplicity,
      then there exists $\varepsilon_0>0$ such that for every $\varepsilon\in(0,\varepsilon_0)$,
      there exists $i_0$ such that
      both $|X_i\cap N_{\varepsilon}(z)|$ and $|X\cap N_{\varepsilon}(z)|$ are finite and equal for every $i\ge i_0$, and
\item if $z$ is an element of $X$ with infinite multiplicity,
      then for every $\varepsilon>0$ and every $k\in\NN$,
      there exists $i_0$ such that
      $|X_i\cap N_{\varepsilon}(z)|\ge k$ for every $i\ge i_0$,
\end{itemize}
We say that a sequence $(A_n)_{n\in\NN}$ of step approximations of a tournamenton $W$ is \emph{strongly convergent}
if it is convergent and the spectra of $A_n$ converge as multisets.
Later, we will show that every convergent sequence of step approximations is also strongly convergent
but we treat the two notions as distinct until we establish their equivalence.
We summarize properties of the limit multiset of a strongly convergent sequence of step approximations in the next proposition.

\begin{proposition}
\label{prop:limitset}
Let $W$ be a tournamenton and
let $(A_n)_{n\in\NN}$ be a strongly convergent sequence of step approximations of $W$.
The limit multiset $X$ of the spectra of $A_n$ satisfies the following.
\begin{itemize}
\item The set $X$ contains zero and its multiplicity is infinite.
\item Every non-zero element of $X$ has finite multiplicity.
\item Every element of $X$ has a non-negative real part.
\item The real parts of the elements of $X$ sum to at most $1/2$ (taking their multiplicities into account).
\item If $x$ is an element of $X$ that is not real, then $X$ contains the complex conjugate of $x$ and
      the multiplicities of $x$ and its complex conjugate are the same.
\item If $X$ contains a non-zero element,
      then it contains a positive real $\rho$ such that the absolute value of all elements of $X$ is at most $\rho$.
\item The sum of the $\ell$-th powers of the elements of $X$ is absolutely convergent for every $\ell\ge 2$.
\item It holds that
      \begin{equation}
      C(W,\ell)=2^{\ell}\cdot\sum_{x\in X}x^\ell\label{eq:limitpower}
      \end{equation}
      for every $\ell\ge 3$.
\end{itemize}
\end{proposition}

\begin{proof}
Let $X_n$ be the spectrum of $A_n$.
By Proposition~\ref{prop:matrix},
all the elements of $X_n$ have non-negative real parts and the sum of their real parts is $1/2$.
Hence, all the elements of $X$ have non-negative real parts and their real parts sum to at most $1/2$.
Since the matrix $A_n$ is real,
every non-real eigenvalue of $A_n$ comes in a pair with a complex conjugate eigenvalue of the same multiplicity.
Hence, the pairs of complex conjugate non-real elements of $X$ must have the same multiplicity.
Let $\rho_n$ be the largest real eigenvalue of $A_n$.
Note that the absolute value of all the elements of $X_n$ is at most $\rho_n$ by Proposition~\ref{prop:matrix}.
The sequence $\rho_n$ converges (because the sets $X_n$ converge as multisets) and its limit $\rho$ belongs to $X$.
If $\rho=0$, then $X$ has no non-zero elements.
If $\rho>0$, then the absolute value of all elements of $X$ is at most $\rho$.

We next show that the sum of the $\ell$-th powers of the elements of $X$ is absolutely convergent for every $\ell\ge 2$.
Let $X^+$ and $X^-$ be the multiset of the non-zero elements $x$ of $X$ such that $\Real x\ge|x|/2$ and $\Real x\le|x|/2$, respectively.
Note that $\Real x^2 \le - |x^2|/2 < 0$ for every $x\in X^-$.
First observe that
\begin{equation}
\sum_{x\in X^+}|x|\le 2\sum_{x\in X^+}\Real x\le 1.\label{eq:X+}
\end{equation}
Since $\Real x^2\le (\Real x)^2$ for every $x\in X_n$ and
the sum of real parts of the elements of $X_n$ is $1/2$,
it follows that
\[\sum_{x\in X_n,\Real x^2>0}\Real x^2\le\frac{1}{4}.\]
Since the trace of $A_n^2$ is non-negative,
we obtain that
\[-\sum_{x\in X_n,\Real x^2<0}\Real x^2\le\frac{1}{4}.\]
In particular, it holds that
\[
-\sum_{x\in X^-}\Real x^2\le\frac{1}{4}.
\]
As $-\Real x^2\ge |x^2|/2$ for every $x\in X^-$, we obtain that
\begin{equation}
\sum_{x\in X^-}|x|^2\le -2\sum_{x\in X^-}\Real x^2\le\frac{1}{2}.\label{eq:X-}
\end{equation}
The inequalities \eqref{eq:X+} and \eqref{eq:X-} yield that
all elements of the multiset $X$ except for $0$ have finite multiplicity and
the sum of the $\ell$-th powers of the elements of $X$ is absolutely convergent for every $\ell\ge 2$.
Since the sizes of the multisets $X_n$ tend to infinity and $|x|\le 1/2$ for all their elements,
the multiset $X$ must contain an element with infinite multiplicity;
since such an element can be only $0$, the multiset $X$ contains $0$ with infinite multiplicity.

It remains to establish \eqref{eq:limitpower}. Fix $\varepsilon\in (0,1)$.
Similarly to the previous paragraph,
define $X_n^+$ to be the multiset of the elements $x$ of $X_n$ such that $\Real x\ge|x|/2$ and
$X_n^-$ to be the multiset of the elements $x$ of $X_n$ such that $\Real x\le|x|/2$.
Along the lines leading to \eqref{eq:X+} and \eqref{eq:X-}, we obtain that
\[\sum_{x\in X_n^+}|x|^2\le\frac{1}{2}\sum_{x\in X_n^+}|x|\le\frac{1}{2}\qquad\mbox{and}\qquad\sum_{x\in X_n^-}|x|^2\le\frac{1}{2}.\]
Hence, we obtain for $\ell \geq 3$ that
\begin{equation}
\left|\Trace A_n^\ell-\sum_{x\in X_n,|x|>\varepsilon}x^\ell\right|
  \le\sum_{x\in X_n,|x|\le\varepsilon}|x|^\ell
  \le\varepsilon\sum_{x\in X_n,|x|\le\varepsilon}|x|^2\le\varepsilon.\label{eq:Xneps}
\end{equation}
Similarly, we obtain that
\begin{equation}
\left|\sum_{x\in X,|x|\le\varepsilon}x^\ell\right|
  \le\sum_{x\in X,|x|\le\varepsilon}|x|^\ell
  \le\varepsilon\sum_{x\in X,|x|\le\varepsilon}|x|^2\le\varepsilon.\label{eq:Xeps}
\end{equation}
Consequently, it follows from \eqref{eq:Xneps} and \eqref{eq:Xeps} that
\begin{equation}
\lim_{n\to\infty}\left|\Trace A_n^\ell-\sum_{x\in X_n}x^\ell\right|\le 2\varepsilon.\label{eq:Xnlim}
\end{equation}
Since the estimate \eqref{eq:Xnlim} holds for every $\varepsilon\in (0,1)$,
It follows that
\[\lim_{n\to\infty}\Trace A_n^\ell=\sum_{x\in X_n}x^\ell.\]
The identity \eqref{eq:limitpower} now follows from Proposition~\ref{prop:cut}.
\end{proof}

By compactness, every convergent sequence of step approximations of $W$ has a strongly convergent subsequence.
The limit multisets of two such strongly convergent subsequences must be the same
since no two different multisets can be absolutely convergent and satisfy \eqref{eq:limitpower} for every $\ell\ge 3$.
Hence, every convergent sequence of step approximations of $W$ is also strongly convergent, 
which we state as a corollary.

\begin{corollary}
\label{cor:strongconv}
Every convergent sequence $(A_n)_{n\in\NN}$ of step approximations of a tournamenton $W$ is strongly convergent.
\end{corollary}

Corollary~\ref{cor:strongconv} implies that
the limit multiset of every convergent sequence of step approximations of a tournamenton $W$ is the same;
we will write $\wsigma(W)$ for this limit multiset after removing $0$.
Proposition~\ref{prop:limitset} now yields that
\begin{equation}
C(W,\ell)=2^{\ell}\cdot\sum_{\lambda\in\wsigma(W)}\lambda^\ell\label{eq:lambda}
\end{equation}
for every $\ell\ge 3$.

We conclude with two propositions relating structural properties of a tournamenton $W$ and $\wsigma(W)$.

\begin{proposition}
\label{prop:reg}
A tournamenton $W$ is regular if and only if $1/2\in\wsigma(W)$.
\end{proposition}

\begin{proof}
Fix a tournamenton $W$.
Let $(A_n)_{n\in\NN}$ be a convergent sequence of step approximations of $W$.
Let $k_n$ be the order of $A_n$,
$\rho_n$ the largest real eigenvalue of $A_n$, and
$v_n$ the corresponding eigenvector with norm one.
Further let $\JJ_n$ be the $k_n\times k_n$ matrix with all entries equal to $k_n^{-1}$ and
$j_n$ the $k_n$-dimensional vector with all entries equal to $k_n^{-1/2}$.
Note that $j_n=\JJ_n j_n$ and $1$ is the only non-zero eigenvalue of $\JJ_n$.
Finally, let $j_0$ be the function $[0,1]\to [0,1]$ such that $j_0(x)=1$ for all $x\in [0,1]$.

If $W$ is a regular tournamenton, then $A_nj_n=j_n/2$.
It follows that $1/2$ is an eigenvalue of $A_n$ for every $n\in\NN$ and so $1/2\in\wsigma(W)$.

We next assume that $1/2\in\wsigma(W)$ and show that the tournamenton $W$ is regular.
Since the matrix $A_n$ cannot have a real eigenvalue larger than $1/2$ by Proposition~\ref{prop:matrix},
it follows that the values of $\rho_n$ converge to $1/2$.
Observe that
\[v_n^T\JJ_nv_n=v_n^T(A_n^T+A_n)v_n=v_n^TA_n^Tv_n+v_n^TA_nv_n=2\rho_n.\]
It follows that $<v_n|j_n>^2=2\rho_n$, which implies that
\[\|v_n-j_n\|=2-2<v_n|j_n>=2-2\sqrt{2\rho_n}.\]
Since $\|A_n\|_F\le 1$, we obtain that
\begin{align*}
\|j_n-2A_nj_n\| &\le\|j_n-v_n\|+\|v_n-2A_nv_n\|+2\|A_nv_n-A_nj_n\|\\
                &\le 3\|j_n-v_n\|+1-2\rho_n = 7-2\rho_n-6\sqrt{2\rho_n}.
\end{align*}		
It follows that
\[\lim_{n\to\infty}\|j_n-2A_nj_n\|=0.\]
Since the step tournamentons $W[A_n]$ converge to the tournamenton $W$ in the cut distance,
we obtain that $\|j_0-2Wj_0\|_2=0$
where $Wj_0$ is the function resulting from applying the linear operator given by $W$ to the function $j_0$.
It follows that the tournamenton $W$ is regular.
\end{proof}

\begin{proposition}
\label{prop:qrand}
A tournamenton $W$ is equal to $1/2$ almost everywhere if and only if $\wsigma(W)=\{1/2\}$.
\end{proposition}

\begin{proof}
Fix a tournamenton $W$.
Let $(A_n)_{n\in\NN}$ be a convergent sequence of step approximations of $W$, and
let $k_n$ be the order of $A_n$.
Further let $\JJ_n$ be the $k_n\times k_n$ matrix with all entries equal to $(2k_n)^{-1}$ and
$j_n$ the $k_n$-dimensional vector with all entries equal to $k_n^{-1/2}$.
Note that the definition of the matrix $\JJ_n$ differs from that in the proof of Proposition~\ref{prop:reg}.
In particular, it holds that $\JJ_nj_n=j_n/2$.

If the tournamenton $W$ is equal to $1/2$ almost everywhere,
then $A_n=\JJ_n$ for every $n\in\NN$ and the only non-zero eigenvalue of $A_n$ is $1/2$.
It follows that $\wsigma(W)=\{1/2\}$.

We next assume that $\wsigma(W)=\{1/2\}$ and show that $W$ is equal to $1/2$ almost everywhere.
Since $1/2\in\wsigma(W)$, the tournamenton $W$ is regular and it follows that $A_nj_n=j_n/2$.
Define $B_n=A_n-\JJ_n$ and observe that $B_n$ is a skew-symmetric matrix and that $B_nj_n$ is the zero vector,
i.e., the vector $j_n$ belongs to the kernel of~$B_n$.
Hence, the non-zero eigenvalues of $A_n$ are $1/2$ and the non-zero eigenvalues of~$B_n$,
which are square roots of the (real and negative) eigenvalues of the symmetric matrix $B_n^2$;
in particular, they all are purely imaginary.

Suppose that $W$ is not equal to $1/2$ almost everywhere,
in particular, the cut distance of $W$ and the tournamenton equal to $1/2$ everywhere is positive.
Since the tournamentons $W[A_n]$ converge to $W$ in the cut distance,
there exists a sequence of vectors $v_n\in\RR^{k_n}$ with $\|v_n\|=1$ and a real $\delta>0$ such that
\[\lim_{n\to\infty}\|B_nv_n\|\ge\delta.\]
In particular,
\[\lim_{n\to\infty} |v_n^TB_n^2v_n|\ge\delta^2,\]
which implies that the smallest eigenvalue $B_n^2$ is less than $-\delta^2$.
Hence, every $B_n$ has a purely imaginary eigenvalue with absolute value at least $\delta$.
However, this is impossible since $\wsigma(W)=\{1/2\}$.
We conclude that $W$ is equal to $1/2$ almost everywhere.
\end{proof}

\section{Cycles of length not divisible by four}
\label{sec:ck12}

In this section, we compute $c(\ell)$ when $\ell$ is not divisible by four and
we characterize tournamentons that are extremal.
The proofs of both Theorem~\ref{thm:ck1} and~\ref{thm:ck2}
are based on the analysis of spectra of linear operators associated with tournamentons,
however,
the arguments apply the same in the setting of tournament matrices.

\begin{theorem}
\label{thm:ck1}
If $\ell\ge 3$ is odd,
then $C(W,\ell)\le 1$ for every tournamenton $W$, and equality holds if and only if $W$ is regular.
In particular, $c(\ell)=1$.
\end{theorem}

\begin{proof}
Fix $\ell\ge 3$ and a tournamenton $W$.
Let $\rho$ be the largest positive real contained in $\wsigma(W)$.
We start with establishing the following.
If $z$ is a complex number with $|z|\le\rho$ and $\Real z\ge 0$,
then
\begin{equation}
\Real z^{\ell}\le \ell \rho^{\ell-1} \Real z, \label{eq:ck1}
\end{equation}
and equality holds if and only if $\Real z=0$.
Consider such $z$ and
let $\alpha$ be such that $\Real z=|z|\cdot\cos\alpha$ and $\Imag z=|z|\cdot\sin\alpha$.
If $\Real z=0$, the estimate \eqref{eq:ck1} holds with equality.
Hence, we can assume that $\alpha\in [0,\pi/2)$ (considering the complex conjugate of $z$ if needed).
If $\ell$ is one modulo four, we set $\beta=\pi/2-\alpha$ and obtain the following:
\[
\Real z^{\ell}=|z|^{\ell}\cos\ell\alpha
              =|z|^{\ell}\sin\ell\beta
	      <\ell|z|^{\ell}\sin\beta
	      =\ell|z|^{\ell}\cos\alpha
	      =\ell|z|^{\ell-1}\Real z.
\]
If $\ell$ is three modulo four, we set $\beta=-\pi/2+\alpha$ and obtain the following:
\[
\Real z^{\ell}=|z|^{\ell}\cos\ell\alpha
              =|z|^{\ell}\sin\ell\beta
	      <-\ell|z|^{\ell}\sin\beta
	      =\ell|z|^{\ell}\cos\alpha
	      =\ell|z|^{\ell-1}\Real z.
\]
We now obtain the estimate \eqref{eq:ck1} using $|z|\le\rho$.

We next bound the sum of the $\ell$-th powers of the elements of $\wsigma(W)$
by applying~\eqref{eq:ck1} to every element of $\wsigma(W)$ except for $\rho$.
Note that we treat $\wsigma(W)$ as a multiset,
i.e., if the multiplicity of $\rho$ in $\wsigma(W)$ is larger than one,
then $\wsigma(W)\setminus\{\rho\}$ contains~$\rho$.
\begin{equation}
  \sum_{\lambda\in\wsigma(W)}\lambda^{\ell}
  =\sum_{\lambda\in\wsigma(W)}\Real\lambda^{\ell}
  \le\rho^\ell+\;\; \sum_{\mathclap{\lambda\in\wsigma(W)\setminus\{\rho\}}}\;\; \ell \rho^{\ell-1}\Real\lambda
  \le\left(\rho+\;\; \sum_{\mathclap{\lambda\in\wsigma(W)\setminus\{\rho\}}}\;\; \Real\lambda\right)^\ell
  \label{eq:ck1s}
\end{equation}  
Since the sum of the real parts of the elements of $\wsigma(W)$ is at most $1/2$ by Proposition~\ref{prop:limitset},
we obtain that
\[\sum_{\lambda\in\wsigma(W)}\lambda^{\ell}\le\frac{1}{2^{\ell}}.\]
The identity \eqref{eq:lambda} now yields that $C(W,\ell)\le 1$ and
since the choice of $W$ was arbitrary, it follows that $c(\ell)\le 1$.
Moreover, if $C(W,\ell)=1$, the sum of real parts of the elements of $\wsigma(W)$ is $1/2$ and
the estimate \eqref{eq:ck1} holds with equality for every element of $\wsigma(W)\setminus\{\rho\}$.
In particular, the real part of every element of $\wsigma(W)\setminus\{\rho\}$ is zero.
Hence, if $C(W,\ell)=1$, then $\rho=1/2$ and the tournamenton $W$ is regular by Proposition~\ref{prop:reg}.
\end{proof}

We next focus on the case when $\ell$ is even but not divisible by four.

\begin{theorem}
\label{thm:ck2}
If $\ell\ge 6$ is even but not divisible by four,
then $C(W,\ell)\le 1$ for every tournamenton $W$ and equality holds if and only if $W$ is equal to $1/2$ almost everywhere.
In particular, $c(\ell)=1$.
\end{theorem}

\begin{proof}
Fix $\ell\ge 6$ and a tournamenton $W$, and
let $\rho$ be the largest positive real contained in $\wsigma(W)$.
We start with establishing the following.
If $z$ is a complex number with $|z|\le\rho$ and $\Real z\ge 0$,
then
\begin{equation}
\Real z^{\ell}\le \ell \rho^{\ell-1} \Real z, \label{eq:ck2}
\end{equation}
and equality holds only if $z=0$.
Consider such $z\not=0$ and
let $\alpha$ be such that $\Real z=|z|\cdot\cos\alpha$ and $\Imag z=|z|\cdot\sin\alpha$.
By symmetry, we can assume that $\alpha\in [0,\pi/2]$.
Let $\beta=\pi/2-\alpha$.
We first show that
\begin{equation}
-\cos\ell\beta<\ell\sin\beta.\label{eq:ck2sin}
\end{equation}
If $0\le\beta<\frac{\pi}{2\ell}$, then $\cos\ell\beta>0$, and
the inequality in \eqref{eq:ck2sin} holds since its left side is negative while the right side is non-negative.
If $\beta\ge\frac{\pi}{2\ell}$, then $\sin\beta>\frac{1}{\ell}$, and
the inequality in \eqref{eq:ck2sin} holds since its right side is larger than one.
We now apply \eqref{eq:ck2sin} as follows.
\[
\Real z^{\ell}=|z|^{\ell}\cos\ell\alpha
              =-|z|^{\ell}\cos\ell\beta
	      <\ell|z|^{\ell}\sin\beta
	      =\ell|z|^{\ell}\cos\alpha
	      =\ell|z|^{\ell-1}\Real z.
\]
We now obtain the estimate \eqref{eq:ck2} using $|z|\le\rho$.

Similarly to the proof of Theorem~\ref{thm:ck2},
we bound the sum of the $\ell$-th powers of the elements of $\wsigma(W)$ using \eqref{eq:ck2}
as follows.
\[
  \sum_{\lambda\in\wsigma(W)}\lambda^{\ell}
  =\sum_{\lambda\in\wsigma(W)}\Real\lambda^{\ell}
  \le\rho^\ell+\sum_{\lambda\in\wsigma(W)\setminus\{\rho\}}\ell \rho^{\ell-1}\Real\lambda
  \le\left(\rho+\sum_{\lambda\in\wsigma(W)\setminus\{\rho\}}\Real\lambda\right)^\ell
\]
Note that the inequality is strict unless $\rho$ is the only non-zero element of $\wsigma(W)$.
Since the sum of the real parts of the elements of $\wsigma(W)$ is at most $1/2$,
we obtain using \eqref{eq:lambda} that
\[C(W,\ell)=2^\ell\sum_{\lambda\in\wsigma(W)}\lambda^{\ell}\le 1.\]
Since the choice of $W$ was arbitrary, it follows that $c(\ell)\le 1$.
Moreover, if $C(W,\ell)=1$, then $\rho$ is the only element of $\wsigma(W)$ and $\rho=1/2$.
Consequently, if $C(W,\ell)=1$, then $W$ is equal to $1/2$ almost everywhere by Proposition~\ref{prop:qrand}.
\end{proof}

\section{Cycles of length divisible by four}
\label{sec:ck4}

The proof of the main result of this section
requires bounding the spectral radius of skew-symmetric matrices with all entries between $-1$ and $+1$.
To get the tight bound on the spectral radius of such matrices,
we need the following auxiliary lemma.

\begin{lemma}
\label{lm:sumsq}
Let $s_1,\ldots,s_k$ be positive reals.
Further let $x_1,\ldots,x_k$ be non-negative reals such that
$x_1\ge x_2\ge \cdots \ge x_k\ge 0$ and
\[
\begin{array}{ccccccccccccccccc}
 x_1 &   &     &   &        &   &     & \le & s_1 & + & s_2 & + & s_3 & + & \cdots & + & s_k\\
 x_1 & + & x_2 &   &        &   &     & \le & s_1 & + & 2s_2 & + & 2s_3 & + & \cdots & + & 2s_k\\
 x_1 & + & x_2 & + & x_3    &   &     & \le & s_1 & + & 2s_2 & + & 3s_3 & + & \cdots & + & 3s_k\\
 \vdots && \vdots &&        & \ddots &&     & \vdots && \vdots & & \vdots & &        &   & \vdots\\
 x_1 & + & x_2 & + & \cdots & + & x_k & \le & s_1 & + & 2s_2 & + & 3s_3 & + & \cdots & + & ks_k,
\end{array}
\]
i.e., the $m$-th inequality, $m\in [k]$ is
\[\sum_{i=1}^m x_i \le \sum_{i=1}^m \min\{i,m\}\cdot s_i.\]
It then holds that
\[\sum_{i=1}^k x_i^2\le\sum_{i=1}^k\left(\sum_{j=i}^k s_j\right)^2\]
and equality holds if and only if 
$x_i=s_i+s_{i+1}+\cdots+s_k$ for all $i\in [k]$.
\end{lemma}

\begin{proof}
Let $S_i$, $i\in [k]$, be the right side of the $i$-th inequality listed in the statement of the lemma, and set $S_0=0$.
Observe that \[S_1-S_0>S_2-S_1>\cdots>S_k-S_{k-1}=s_k.\]
Since the set of reals $x_1\ge x_2\ge \cdots \ge x_k\ge 0$ that
satisfy the $k$ inequalities $x_1+\cdots+x_i\le S_i$, $i\in [k]$, is compact,
there exists a $k$-tuple $x_1,\ldots,x_k$ that maximizes the sum $x_1^2+\cdots+x_k^2$
subject to $x_1\ge x_2\ge \cdots \ge x_k\ge 0$ and the $k$ inequalities.
Fix such a $k$-tuple. Observe that $x_k\ge s_k>0$.

We first establish that $x_1>x_2>\cdots>x_k$.
Suppose that there exists $i$ such that $x_i=x_{i+1}$. Choose the smallest such $i$,
and let $i'$ be the largest index such that $x_{i'}=x_i$.
Next suppose that $x_1+\cdots+x_j=S_j$ for some $j\in\{i,\ldots,i'-1\}$.
Since $x_1+\cdots+x_{j-1}\le S_{j-1}$, it follows that $x_j\ge S_j-S_{j-1}$.
We next obtain using that $x_j=x_{j+1}$ the following:
\[x_1+\cdots+x_j+x_{j+1}=x_1+\cdots+x_j+x_j\ge S_j+(S_j-S_{j-1})>S_j+(S_{j+1}-S_j)=S_{j+1},\]
which violates the $(j+1)$-th inequality listed in the statement of the lemma.
Hence, it holds that $x_1+\cdots+x_j<S_j$ for every $j=i,\ldots,i'-1$.
Choose $\varepsilon>0$ such that
\begin{itemize}
\item $x_1+\cdots+x_j+\varepsilon\le S_j$ for every $j=i,\ldots,i'-1$,
\item if $i\ge 2$, then $x_i+\varepsilon\le x_{i-1}$,
\item if $i'\le k-1$, then $x_{i'}-\varepsilon\ge x_{i'+1}$, and
\item if $i'=k$, then $x_{i'}-\varepsilon\ge 0$.
\end{itemize}
Consider $x'_1,\ldots,x'_k$ such that
\[x_j=\begin{cases}
      x_j+\varepsilon & \mbox{if $j=i$,} \\
      x_j-\varepsilon & \mbox{if $j=i'$, and} \\
      x_j & \mbox{otherwise.}
      \end{cases}\]
Observe that $x'_1\ge x'_2\ge\cdots\ge x'_k\ge 0$ and $x'_1+\cdots+x'_j\le S_j$ for every $j\in [k]$.
Since the sum of the squares of $x'_1,\ldots,x'_k$ is larger than the sum of the squares of $x_1,\ldots,x_k$,
we obtain that the $k$-tuple $x_1,\ldots,x_k$ does not maximize the sum of the squares subject
to $x_1\ge x_2\ge \cdots \ge x_k\ge 0$ and the $k$ inequalities listed in the statement of the lemma.
This contradicts the choice of $x_1,\ldots,x_k$.
Hence, we have established that $x_1>x_2>\cdots>x_k$.

We next show that $x_1+\cdots+x_i=S_i$ for every $i\in [k]$.
Suppose the opposite, that $x_1+\cdots+x_i<S_i$ for some $i$, and choose $\varepsilon>0$ such that
\begin{itemize}
\item $x_1+\cdots+x_i+\varepsilon\le S_i$,
\item if $i\ge 2$, then $x_i+\varepsilon\le x_{i-1}$,
\item if $i\le k-2$, then $x_{i+1}-\varepsilon\ge x_{i+2}$, and
\item if $i=k-1$, then $x_{i+1}-\varepsilon\ge 0$.
\end{itemize}
Consider $x'_1,\ldots,x'_k$ such that
\[x_j=\begin{cases}
      x_j+\varepsilon & \mbox{if $j=i$,} \\
      x_j-\varepsilon & \mbox{if $j=i+1$, and} \\
      x_j & \mbox{otherwise.}
      \end{cases}\]
Observe that $x'_1\ge x'_2\ge\cdots\ge x'_k\ge 0$ and $x'_1+\cdots+x'_j\le S_j$ for every $j\in [k]$.
Since the sum of the squares of $x'_1,\ldots,x'_k$ is larger than the sum of the squares of $x_1,\ldots,x_k$,
we obtain that the $k$-tuple $x_1,\ldots,x_k$ does not maximize the sum of the squares subject
to $x_1\ge x_2\ge \cdots \ge x_k\ge 0$ and the $k$ inequalities listed in the statement of the lemma.
Hence, we conclude that $x_1+\cdots+x_i=S_i$ for every $i\in [k]$,
which implies that $x_i=S_i-S_{i-1}$ for every $i\in [k]$.
Since it holds that $S_i-S_{i-1}=s_i+\cdots+s_k$, the statement of the lemma now follows.
\end{proof}

We next bound the spectral radius of skew-symmetric matrices with entries between $-1$ and $+1$.
For $n\in\NN$, define $D_n$ to be the skew-symmetric matrix
with all entries above the diagonal equal to $+1$ and all entries below the diagonal equal to $-1$.
The next lemma asserts that the matrix $D_n$ has the largest possible spectral radius
among all skew-symmetric matrices with entries between $-1$ and~$+1$.

\begin{lemma}
\label{lm:antisym}
For every $n\in\NN$,
the spectral radius of a skew-symmetric matrix $A\in [-1,1]^{n\times n}$
is at most the spectral radius of $D_n$.
\end{lemma}

\begin{proof}
We fix $n$ and write $D$ for $D_n$ throughout the proof.
We establish that
for every vector $v\in\RR^n$ with $\|v\|=1$, there exists a vector $w$ that can be obtained from $v$
by permuting the entries of $v$ and changing their signs such that $\|Av\|\le\|Dw\|$.
This would imply that the spectral radius of $A$ does not exceed that of $D$.
First observe that it is enough to prove the inequality for vectors $v\in\RR^n$ with non-zero entries (if the statement fails
for a vector $v$, it also fails for any unit vector obtained by any small perturbation of $v$).
Next observe that if $A'$ is obtained by changing the signs of all entries in the $i$-th row and the $i$-th column and
$v'$ is obtained from $v$ by changing the sign of its $i$-th entry,
then $(Av)_i=-(A'v')_i$ and $(Av)_j=(A'v')_j$ for all $j\not=i$; in particular, $\|Av\|=\|A'v'\|$.
Hence, we will assume without loss of generality that all entries of $v$ are positive.

By permuting rows and columns of $A$ symmetrically and applying the same permutation to $v$,
we can assume that $(Av)_1\ge (Av)_2\ge \cdots \ge (Av)_n$.
Let $k$ be the largest index such that $(Av)_k\ge 0$, and
let $w$ be the vector obtained from $v$ by permuting its first $k$ entries and its remaining $n-k$ entries separately in a way that
$w_1\le w_2\le\dots\le w_k$ and $w_{k+1}\ge w_{k+2}\ge\dots\ge w_n$.
We will show that
\begin{equation}
\sum_{i=1}^k(Av)_i^2\le\sum_{i=1}^k(Dw)_i^2\quad\mbox{and}\quad\sum_{i=k+1}^n(Av)_i^2\le\sum_{i=k+1}^n(Dw)_i^2.\label{eq:antisym1}
\end{equation}
The arguments for the two cases are symmetric and so we focus on establishing that
\begin{equation}
\sum_{i=1}^k(Av)_i^2\le\sum_{i=1}^k(Dw)_i^2.\label{eq:antisym2}
\end{equation}
Observe that the following holds for every $m\in [k]$ (we use that $A_{ij}+A_{ji}=0$ for all $i,j\in [m]$):
\begin{align*}
\sum_{i=1}^m (Av)_i & = \sum_{i=1}^m\sum_{j=1}^n A_{ij}v_j \\
                    & \le \sum_{1\le i<j\le m}\max\{v_i-v_j,v_j-v_i\}+\sum_{i=1}^m\sum_{j=m+1}^n v_j \\
                    & = \sum_{1\le i<j\le m} \left(v_i+v_j-2\min\{v_i,v_j\}\right)+m\sum_{j=m+1}^n v_j \\
                    & = -\sum_{1\le i,j\le m} \min\{v_i,v_j\}+m\sum_{j=1}^n v_j \\
                    & \le -\sum_{1\le i,j\le m} \min\{w_i,w_j\}+m\sum_{j=1}^n w_j \\
		    & = \sum_{1\le i<j\le m} \left(w_j-w_i\right)+m\sum_{j=m+1}^n w_j \\
                    & = \sum_{i=1}^m\sum_{j=1}^n D_{ij}w_j = \sum_{i=1}^m (Dw)_i.
\end{align*}
Let $k'$ be the largest index such that $k'\le k$ and $(Dw)_{k'}>0$.
We choose $\varepsilon>0$ and apply Lemma~\ref{lm:sumsq} with the following parameters:
$x_i=(Av)_i$, $i\in [k]$, 
$s_i=(Dw)_i-(Dw)_{i+1}$ for $i\in [k'-1]$, $s_{k'}=(Dw)_{k'}$ and $s_i=\varepsilon$, $i\in [k]\setminus [k']$.
Observe that $x_1,\ldots,x_{k}$ and $s_1,\ldots,s_{k}$ satisfy the assumptions of Lemma~\ref{lm:sumsq}.
Hence,
Lemma~\ref{lm:sumsq} implies that
\[\sum_{i=1}^{k} x_i^2\le\sum_{i=1}^{k'}\left((Dw)_i+(k-k')\varepsilon\right)^2+(k-k')\varepsilon^2.\]
Since this inequality holds for every $\varepsilon>0$,
we obtain that
\[\sum_{i=1}^{k} x_i^2\le\sum_{i=1}^{k'}(Dw)_i^2.\]
This establishes \eqref{eq:antisym2}.
The other inequality in \eqref{eq:antisym1} can be proven analogously.
Hence, we conclude that $\|Av\|\le\|Dw\|$ as desired.
\end{proof}

We next use Lemma~\ref{lm:antisym} to bound elements of $\wsigma(W)$ of a regular tournamenton~$W$.

\begin{lemma}
\label{lm:ck4}
Let $W$ be a tournamenton.
If $1/2$ is contained in $\wsigma(W)$,
then each other element of $\wsigma(W)$ has absolute value at most $1/\pi$.
\end{lemma}

\begin{proof}
Let $A_n$, $n\in\NN$, be a convergent sequence of step approximations of $W$, and
let $k_n$ be the order of $A_n$.
Since $W$ is regular by Proposition~\ref{prop:reg},
the sum of each row of $A_n$ is $1/2$.
This yields that $1/2$ is an eigenvalue of $A_n$ and the associated eigenvector is $(1,\ldots,1)$.

Let $J_{k_n}$ be the square matrix of order $k_n$ with all entries equal to $1$ and
let $B_n=J_{k_n}-2k_n\cdot A_n$.
Note that the matrix $B_n$ is skew-symmetric and all its entries are between $-1$ and $+1$.
Also note that the vector $(1,\ldots,1)$ is an eigenvector of $A_n$ associated with the eigenvalue $1/2$, and
it is also an eigenvector of $B_n$ associated with the eigenvalue $0$.
Since $B_n$ is skew-symmetric,
all its non-zero eigenvalues are purely imaginary and
there exists an orthonormal basis of~$\CC^{k_n}$ formed by eigenvectors of the matrix $B_n$
(see the beginning of Section~\ref{sec:c8} for a more detailed exposition on properties of skew-symmetric matrices).
This means that we can assume that every eigenvector of $B_n$ associated with a non-zero eigenvalue
is orthogonal to the vector $(1,\ldots,1)$ in the space $\CC^{k_n}$.

Since $(1,\ldots,1)$ is an eigenvector of $B_n$ associated with the eigenvalue $0$,
every eigenvector of~$B_n$ associated with a non-zero eigenvalue of $B_n$ is also an eigenvector of $A_n$.
Thus, if $\lambda$ is an eigenvalue of $A_n$,
then either $\lambda$ is equal to $1/2$ (for the vector $(1,\ldots,1)$) or $-2k_n\lambda$ is an eigenvalue of~$B_n$.
Since the spectral radius of $B_n$ is at most the spectral radius of the matrix~$D_{k_n}$ by Lemma~\ref{lm:antisym} and
the spectral radiuses of the matrices $D_{k_n}$ divided by $k_n$ converge to $2/\pi$,
it follows that the limit of the maximum absolute value of an eigenvalue of $A_n$ different from $1/2$ is at most $1/\pi$.
The statement of the lemma now follows.
\end{proof}

We are now ready to asymptotically determine $c(\ell)$ for $\ell$ divisible by four.
Since the multiset $\wsigma(W_C)$ for the carousel tournamenton $W_C$,
which we described at the end of Section~\ref{sec:prelim},
consists of $1/2$ and $\pm\unit/((2i-1)\pi)$ for $i\in\NN$,
we obtain using \eqref{eq:lambda} that
\[c(\ell)\ge 1+2\cdot\sum_{i=1}^{\infty}\left(\frac{2}{(2i-1)\pi}\right)^{\ell}\]
for every $\ell$ divisible by four.
The next theorem of this section provides an asymptotically matching upper bound.

\begin{theorem}
\label{thm:ck4}
For every $\varepsilon>0$, there exists $\ell_0$ such that
the following holds for every $\ell\ge\ell_0$ divisible by four:
\[c(\ell)\le 1+\left(\frac{2}{\pi}+\varepsilon\right)^{\ell}.\]
\end{theorem}

\begin{proof}
Let $W_k$ be a tournamenton that maximizes the sum of the $(4k)$-th powers of $\wsigma(W_k)$.
Suppose that the statement of the theorem is false,
i.e., there exists $\varepsilon>0$ and a sequence $(k_i)_{i\in\NN}$ such that
\begin{equation}\label{eq:ck4c}
C(W_{k_i},4k_i)>1+\left(\frac{2}{\pi}+\varepsilon\right)^{4k_i}
\end{equation}
for every $i\in\NN$.
Without loss of generality,
we may assume that the sequence $(W_{k_i})_{i\in\NN}$ is convergent in the cut distance and
let $W$ be the tournamenton that is its limit.
Since partitions of $[0,1]$ corresponding to fine enough step approximations of $W_{k_i}$
yield step approximations close to $W$ in the cut distance and
the step approximations of $W_{k_i}$ and $W$ are also close in the cut distance,
the sets $\left(\wsigma(W_{k_i})\cup\{0^\infty\}\right)_{i\in\NN}$
converge to $\wsigma(W)\cup\{0^\infty\}$ as multisets.

Let $\rho$ be the largest positive real contained in $\wsigma(W)$.
If $\rho$ is smaller than $1/2$, then it holds that
\[\lim_{i\to\infty}C(W_{k_i},4k_i)=0.\]
Since this contradict the choice of $(k_i)_{i\in\NN}$,
we can assume that $\rho=1/2$, i.e., $\wsigma(W)$ contains $1/2$.
Hence, 
all other non-zero elements of $\wsigma(W)$ are purely imaginary and
the absolute value of every such element at most $1/\pi$ by Lemma~\ref{lm:ck4}.
Let $m$ be the number elements of $\wsigma(W)$ with the absolute value $1/\pi$ (note that $m$ can be zero).
We can assume that $\varepsilon$ is small enough that
the absolute value of any element of $\wsigma(W)$ with absolute value smaller than $1/\pi$ is at most $1/\pi-\varepsilon$.
It follows that
\begin{equation}
\lim_{i\to\infty}\frac{C(W_{k_i},4k_i)-2^{4k_i}\sum_{\lambda\in\wsigma(W_{k_i}),|\lambda|\ge 1/\pi-\varepsilon/2}\lambda^{4k_i}}{(2/\pi)^{4k_i}}=0.\label{eq:ck4a}
\end{equation}
The choice of $m$ implies that there exists $i_0$ such that
\begin{equation}
\sum_{\lambda\in\wsigma(W_{k_i}),|\lambda|\ge 1/\pi-\varepsilon/2}\lambda^{4k_i}\le\frac{1}{2^{4k_i}}+m\left(\frac{1}{\pi}+\frac{\varepsilon}{4}\right)^{4k_i}\label{eq:ck4b}
\end{equation}
for every $i\ge i_0$.
Using \eqref{eq:ck4c}, \eqref{eq:ck4a} and \eqref{eq:ck4b} we get a contradiction.
\end{proof}

\section{Cycles of length eight}
\label{sec:c8}

In this section, we present our results on cycles of length eight.
We also present the analogous arguments in the (simpler) case of cycles of length four to make the exposition more accessible,
although the presented results for cycles of length four have been previously proven.
To be able to present our arguments, we need to recall some results on matrices and particularly on skew-symmetric matrices. 
If $A$ is a square matrix of order $n$ and $X\subseteq [n]$,
then $A[X]$ is the square matrix formed by entries in the rows and the columns indexed by the elements of $X$.
Throughout this section, $J_n$ denotes the square matrix of order $n$ with all entries equal to $1$.
Recall that $D_n$ is the skew-symmetric $n\times n$ matrix with all entries above the diagonal equal to $+1$ and all entries below the diagonal equal to $-1$.
We say that two skew-symmetric matrices are \emph{sign-equivalent}
if one can be obtained from the other by permuting the rows and columns symmetrically and
multiplying some of the rows and the symmetric set of columns by $-1$.

It is well-known that for every real skew-symmetric matrix~$A$,
there exists an orthogonal (real) matrix $Q$
(i.e. a square matrix $Q$ such that $Q^TQ$ is the identity matrix) such that
the matrix $Q^TAQ$ is a block diagonal matrix with two kinds of blocks:
blocks of size two of the form $\begin{pmatrix} 0 & a \\ -a & 0 \end{pmatrix}$ and
blocks of size one equal to the zero matrix.
In particular, there exists an orthogonal basis formed by eigenvectors of the matrix $A^2$ and
the values $-a^2$ from the blocks of the matrix $Q^TAQ$ are the eigenvalues of $A^2$ (each with multiplicity two).
In addition, if $v$ and $v'$ are two rows of $Q$ associated with the single block of $Q^TAQ$ with a value $a$,
i.e., $av=Av'$ and $-av'=Av$, the complex vectors $v+\unit v'$ and $v-\unit v'$
are eigenvectors of $A$ associated with the eigenvalues $a\unit$ and $-a\unit$, respectively, and
the vectors $v+\unit v'$ and $v-\unit v'$ are orthogonal in the space $\CC^n$.

We apply the just reviewed results on skew-symmetric matrices to get the following upper bound
the trace of the fourth and eight powers of the sum of the all-one matrix and a skew-symmetric matrix.

\begin{lemma}
\label{lm:midterms}
Let $B$ be a skew-symmetric matrix of order $n$ with entries between $-1$ and $+1$.
It holds that
\begin{align*}
\Trace (J_n+B)^4 &=\Trace J_n^4+\Trace B^4-4n||Bj||^2 \quad\mbox{and}\\
\Trace (J_n+B)^8 &\le\Trace J_n^8+\Trace B^8-2n^5||Bj||^2,
\end{align*}
where $j$ is the vector with all entries equal to one.
In particular, it holds that
\[\Trace (J_n+B)^4\le\Trace J_n^4+\Trace B^4\quad\mbox{and}\quad \Trace (J_n+B)^8\le\Trace J_n^8+\Trace B^8,\]
and equality holds if and only if the sum of each row of $B$ is zero.
\end{lemma}

\begin{proof}
We start with the trace of the $4$-th power of $J_n+B$.
We obtain the following by expanding $(J_n+B)^4$ and using that $\Trace XY=\Trace YX$:
\[\Trace (J_n+B)^4=\Trace J_n^4+4\Trace J_n^3B+4\Trace J_n^2B^2+2\Trace J_nBJ_nB+4\Trace J_nB^3+\Trace B^4.\]
Since $B$ is skew-symmetric, any odd power of $B$ is also skew-symmetric.
In particular, $J_nB^{2k-1}J_n$ is the zero matrix and $\Trace J_n B^{2k-1}=0$ for every $k\in\NN$.
It follows that
\begin{equation}
\Trace (J_n+B)^4=\Trace J_n^4+4\Trace J_n^2B^2+\Trace B^4.\label{eq:mid1}
\end{equation}
Let $Q$ be the orthogonal matrix such that $Q^TBQ$ has the block structure described before the statement of this lemma,
let $k$ be the number of blocks of size two, and
let $a_1,\ldots,a_k$ be the (non-zero) numbers associated with these blocks.
Since the trace of $B^2$ is equal to $2(a_1^2+\cdots+a_k^2)$, it follows that 
\begin{equation}
a_1^2+\cdots+a_k^2\le \frac{n^2}{2}.\label{eq:mida}
\end{equation}
Further, let $q_i$ and $q'_i$ be the two rows of $Q$ corresponding to the block with $a_i$, $i\in [k]$, and
let $\alpha_i\in [0,\pi/2]$ be the angle between the vector $j$ and the plane generated by $q_i$ and $q'_i$, $i\in [k]$.
Note that the $(n-2k)$-dimensional subspace orthogonal to the space generated by $q_1,\ldots,q_k$ and $q'_1,\ldots,q'_k$
is the kernel of $B$ (as it is generated by the rows of $Q$ corresponding to the blocks of size one).
Since the rows of $Q$ form an orthogonal basis, it follows that
\begin{equation}
\sum_{i=1}^k\cos^2\alpha_i\le 1.\label{eq:mid2}
\end{equation}
Observe that the following identities hold.
\begin{align*}
\Trace J_n^4 & = n^4 \\
\Trace J_n^2B^2 & = -n^2\sum_{i=1}^k a_i^2\cos^2\alpha_i \\
\Trace B^4 & = 2\sum_{i=1}^k a_i^4
\end{align*}
In particular, the second term in \eqref{eq:mid1} is non-positive and equal to $-4n||Bj||^2$.
Hence, $\Trace (J_n+B)^4\le \Trace J_n^4+\Trace B^4$ and
equality holds if and only if the vector $j$ is in the kernel of $B$.
The latter holds if and only if the sum of each row of $B$ is zero.
This establishes the statement of the lemma concerning the trace of the $4$-th power of $J_n+B$.

We next analyze the trace of the $8$-th power of $J_n+B$.
As in the case of the $4$-th power the trace of some terms in the expansion of $(J_n+B)^8$ is zero, and
we obtain the following.
\begin{align}
\Trace (J_n+B)^8 & = \Trace J_n^8+8\Trace J_n^6B^2+8\Trace J_n^4B^4+8\Trace J_n^3B^2J_nB^2\nonumber\\
                 & + 4\Trace J_n^2B^2J_n^2B^2+8\Trace J_n^2B^6+8\Trace J_nB^2J_nB^4+\Trace B^8\label{eq:mid3}
\end{align}
As in the previous case, we can express some of the terms in \eqref{eq:mid3} using $a_i$ and $\alpha_i$, $i\in [k]$.
\begin{align*}
\Trace J_n^8 & = n^8 \\
\Trace J_n^6B^2 & = -n^6\sum_{i=1}^k a_i^2\cos^2\alpha_i \\
\Trace J_n^4B^4 & = n^4\sum_{i=1}^k a_i^4\cos^2\alpha_i \\
\Trace J_n^3B^2J_nB^2 & = \Trace J_n^2B^2J_n^2B^2 = n^4\left(\sum_{i=1}^k a_i^2\cos^2\alpha_i\right)^2 \\
\Trace J_n^2B^6 & = -n^2\sum_{i=1}^k a_i^6\cos^2\alpha_i \\
\Trace J_nB^2J_nB^4 & = -n^2\left(\sum_{i=1}^k a_i^2\cos^2\alpha_i\right)\left(\sum_{i=1}^k a_i^4\cos^2\alpha_i\right)\\
\Trace B^8 & = 2\sum_{i=1}^k a_i^8
\end{align*}
We derive using \eqref{eq:mid2} that
% \begin{align}
% 2\Trace J_n^6B^2+8\Trace J_n^3B^2J_nB^2+8\Trace J_nB^2J_nB^4&=\nonumber\\
% -2n^2\left(\sum_{i=1}^k a_i^2\cos^2\alpha_i\right)\left(n^4-4n^2\sum_{i=1}^k a_i^2\cos^2\alpha_i+4\sum_{i=1}^k a_i^4\cos^2\alpha_i\right)&\le\nonumber\\
% -2n^2\left(\sum_{i=1}^k a_i^2\cos^2\alpha_i\right)\left(n^4\sum_{i=1}^k\cos^2\alpha_i-4n^2\sum_{i=1}^k a_i^2\cos^2\alpha_i+4\sum_{i=1}^k a_i^4\cos^2\alpha_i\right)&=\nonumber\\
% -2n^2\left(\sum_{i=1}^k a_i^2\cos^2\alpha_i\right)\left(\sum_{i=1}^k\left(n^2-2a_i^2\right)^2\cos^2\alpha_i\right)&\le 0.\label{eq:mid4+}
% \end{align}
\begin{align}
&2\Trace J_n^6B^2+8\Trace J_n^3B^2J_nB^2+8\Trace J_nB^2J_nB^4\nonumber\\
&=-2n^2\left(\sum_{i=1}^k a_i^2\cos^2\alpha_i\right)\left(n^4-4n^2\sum_{i=1}^k a_i^2\cos^2\alpha_i+4\sum_{i=1}^k a_i^4\cos^2\alpha_i\right)\nonumber\\
&\le-2n^2\left(\sum_{i=1}^k a_i^2\cos^2\alpha_i\right)\left(n^4\sum_{i=1}^k\cos^2\alpha_i-4n^2\sum_{i=1}^k a_i^2\cos^2\alpha_i+4\sum_{i=1}^k a_i^4\cos^2\alpha_i\right)\nonumber\\
&=-2n^2\left(\sum_{i=1}^k a_i^2\cos^2\alpha_i\right)\left(\sum_{i=1}^k\left(n^2-2a_i^2\right)^2\cos^2\alpha_i\right)\le 0.\label{eq:mid4+}
\end{align}
Using \eqref{eq:mida} and \eqref{eq:mid2}, we obtain that
\[\sum_{i=1}^k a_i^2\cos^2\alpha_i\le\frac{n^2}{2},\]
which yields that
\begin{align}
2\Trace J_n^6B^2+4\Trace J_n^2B^2J_n^2B^2&=\nonumber\\
2n^4\left(\sum_{i=1}^k a_i^2\cos^2\alpha_i\right)\left(-n^2+2\sum_{i=1}^k a_i^2\cos^2\alpha_i\right)&\le 0.\label{eq:mid4}
\end{align}
We finally bound a portion of the second term, and the third and sixth terms in \eqref{eq:mid3} as follows.
\begin{align}
2\Trace J_n^6B^2+8\Trace J_n^4B^4+8\Trace J_n^2B^6&=\nonumber\\
-2n^2\sum_{i=1}^k\left(n^4-4n^2a_i^2+4a_i^4\right)a_i^2\cos^2\alpha_i&=\nonumber\\
-2n^2\sum_{i=1}^k\left(n^2-2a_i^2\right)^2a_i^2\cos^2\alpha_i&\le 0\label{eq:mid5}
\end{align}
Using \eqref{eq:mid4+}, \eqref{eq:mid4} and \eqref{eq:mid5},
we obtain the following estimate on $\Trace (J_n+B)^8$ using the expansion in \eqref{eq:mid3}.
\begin{equation}
\Trace (J_n+B)^8 \le \Trace J_n^8+\Trace B^8+2\Trace J_n^6B^2\label{eq:mid6}
\end{equation}
Since it holds that $\Trace J_n^6B^2=-n^5||Bj||^2\le 0$,
the inequality from the statement of the lemma now follows.
Hence, $\Trace (J_n+B)^8\le \Trace J_n^8+\Trace B^8$.
Moreover, since the inequalities \eqref{eq:mid4+}, \eqref{eq:mid4} and \eqref{eq:mid5}
hold with equality if $\cos\alpha_i=0$ for every $i\in [k]$, which holds if $j$ is in the kernel of $B$,
we can conclude that
$\Trace (J_n+B)^8=\Trace J_n^8+\Trace B^8$ if and only if the sum of each row of $B$ is zero.
\end{proof}

The following two lemmas will be important
to establish an upper bound on the trace of the last term in the upper bound given in Lemma~\ref{lm:midterms}.
To state the lemmas, we need the following definition.
For a square matrix $A$ of order $n$ we define the \emph{cyclic index} of $A$ as
\[\Cycl A=\sum_{\pi\in S_n}\prod_{i=1}^nA_{\pi(i)\pi(i+1)},\]
where the computation with indices is modulo $n$, i.e., $\pi(n+1)=\pi(1)$.

\begin{lemma}
\label{lm:trace4}
Let $B$ be a skew-symmetric matrix of order $4$ such that
each off-diagonal entry of $B$ is $+1$ or $-1$.
The cyclic index of $B$ is at most the cyclic index of $D_4$ and
equality holds if and only if $B$ is sign-equivalent to $D_4$.
\end{lemma}

\begin{proof}
Since the cyclic index of sign-equivalent matrices is the same,
we may assume without loss of generality that the first row of $B$ contains $+1$ only.
Hence, we need to consider the following two matrices (after a permutation of rows and columns):
\[
\begin{pmatrix}
  0 & +1 & +1 & +1 \\
 -1 &  0 & +1 & +1 \\
 -1 & -1 &  0 & +1 \\
 -1 & -1 & -1 &  0 \\
\end{pmatrix} 
\quad\mbox{and}\quad
\begin{pmatrix}
  0 & +1 & +1 & +1 \\
 -1 &  0 & +1 & -1 \\
 -1 & -1 &  0 & +1 \\
 -1 & +1 & -1 &  0 \\
\end{pmatrix} 
\]
The cyclic index of the left matrix is $8$ and
the cyclic index of the right matrix is $-24$.
The statement of the lemma follows.
\end{proof}

To state the next lemma, we need to introduce another skew-symmetric matrix.
\[D'_8=
  \begin{pmatrix}
   0 & +1 & +1 & +1 &  +1 & +1 & +1 & +1 \\
  -1 &  0 & +1 & -1 &  +1 & +1 & +1 & +1 \\
  -1 & -1 &  0 & -1 &  +1 & +1 & +1 & +1 \\
  -1 & +1 & +1 &  0 &  -1 & -1 & -1 & -1 \\
  -1 & -1 & -1 & +1 &   0 & +1 & +1 & -1 \\
  -1 & -1 & -1 & +1 &  -1 &  0 & +1 & +1 \\
  -1 & -1 & -1 & +1 &  -1 & -1 &  0 & +1 \\
  -1 & -1 & -1 & +1 &  +1 & -1 & -1 &  0
  \end{pmatrix}
\]

\begin{lemma}
\label{lm:trace8}
Let $B$ be a skew-symmetric matrix of order $8$ such that
each off-diagonal entry of $B$ is $+1$ or $-1$.
The cyclic index of $B$ is at most the cycle index of $D_8$ and
equality holds if and only if $B$ is sign-equivalent to $D_8$ or to $D'_8$.
\end{lemma}

The proof of Lemma~\ref{lm:trace8} proceeds by a computer assisted inspection of all skew-symmetric $8 \times 8$ matrices
where the off-diagonal entries in the first row are $+1$, those in the first column are $-1$, and 
all other off-diagonal entries are either $+1$ or $-1$.
In an independent way, we have prepared a C program and a C++ program to verify Lemma~\ref{lm:trace8},
i.e., to check that the cyclic index of every skew-symmetric $8 \times 8$ matrix of the above form is at most $2\,176$ and
the equality holds if and only if the matrix is sign-equivalent to $D_8$ or to $D'_8$;
the code of the C program and its output are available as ancillary files on arXiv.

We are now ready to compute the value of $c(8)$.

\begin{theorem}
\label{thm:c8}
It holds that $c(4)=4/3$ and $c(8)=332/315$.
\end{theorem}

\begin{proof}
Fix $\ell\in\{4,8\}$ and
let $A$ be the tournament matrix of an $n$-vertex tournament $T$.
Proposition~\ref{prop:eigen} yields that
\[C(T,\ell)=\frac{2^\ell}{n^\ell}\Trace A^\ell+O(n^{-1}).\]
Let $B=J_n-2A$ and note that $A=\frac{J_n+B}{2}$.
By Lemma~\ref{lm:midterms}, we obtain that
\[\Trace A^\ell\le\frac{1}{2^\ell}\left(\Trace J_n^\ell+\Trace B^\ell\right).\]
The trace of $B^\ell$ can be combinatorially interpreted as the sum taken over all closed walks with length $\ell$ in $T$
where the sum contains $+1$ for every such walk with an even number of forward edges and
$-1$ for every such walk with an odd number of forward edges.
Such walks that are not cycles contribute to the sum only $O(n^{\ell-1})$ and
those that are cycles can be counted as cyclic indices of the square submatrices
with rows and columns indexed by the vertices of the cycle.
Hence, we obtain the following.
\[\Trace B^\ell=\sum_{X\in\binom{[n]}{\ell}}\Cycl B[X]+O(n^{\ell-1}).\]
By Lemmas~\ref{lm:trace4} and~\ref{lm:trace8},
it holds $\Cycl B[X]\le\Cycl D_\ell$ for every $X\in\binom{[n]}{\ell}$,
which yields that
\[\Trace B^\ell\le\Trace D_n^\ell+O(n^{\ell-1}).\]
Hence, we obtain that 
\[C(T,\ell)\le\frac{1}{n^\ell}\left(n^\ell+\Trace D_n^\ell+O(n^{\ell-1})\right).\]

We next proceed separately for $\ell=4$ and $\ell=8$.
Analyzing the spectrum of the matrix $D_n$ yields that
\[\lim_{n\to\infty}\frac{\Trace D_n^4}{n^4}=2\sum_{i=1}^{\infty}\left(\frac{2}{(2i-1)\pi}\right)^4=\frac{1}{3},\]
which implies that
\[C(T,4)\le\frac{4}{3}+O(n^{-1}).\]
Similarly, we obtain that
\[\lim_{n\to\infty}\frac{\Trace D_n^8}{n^8}=2\sum_{i=1}^{\infty}\left(\frac{2}{(2i-1)\pi}\right)^8=\frac{17}{315},\]
which implies that
\[C(T,8)\le\frac{332}{315}+O(n^{-1}).\]
Since it holds that $C(W_C,4)=4/3$ and $C(W_C,8)=332/315$ for the carousel tournamenton $W_C$,
the statement of the theorem now follows.
\end{proof}

The methods used to prove Theorem~\ref{thm:c8} actually provide the characterization of extremal tournamentons.
We fix some additional notation:
$T^4$ is the $4$-vertex transitive tournament,
$C^4$ is the unique $4$-vertex hamiltonian tournament,
$L^4$ is the unique $4$-vertex non-transitive tournament with a sink, and
$W^4$ is the unique $4$-vertex non-transitive tournament with a source.
The four tournaments are depicted in Figure~\ref{fig:CLTW}.
It is also interesting to note that the matrix $D'_8$ is sign-equivalent to the following matrix $D''_8$:
\[D''_8=
  \begin{pmatrix}
   0 & +1 & +1 & -1 &  +1 & +1 & +1 & +1 \\
  -1 &  0 & +1 & +1 &  +1 & +1 & +1 & +1 \\
  -1 & -1 &  0 & +1 &  +1 & +1 & +1 & +1 \\
  +1 & -1 & -1 &  0 &  +1 & +1 & +1 & +1 \\
  -1 & -1 & -1 & -1 &   0 & +1 & +1 & -1 \\
  -1 & -1 & -1 & -1 &  -1 &  0 & +1 & +1 \\
  -1 & -1 & -1 & -1 &  -1 & -1 &  0 & +1 \\
  -1 & -1 & -1 & -1 &  +1 & -1 & -1 &  0
  \end{pmatrix}
  .
\]
The $8$-vertex tournament with the tournament matrix $\frac{J_8+D''_8}{2}$
is the tournament obtained from two copies of $C^4$ by adding edges directed from the first copy to the second;
this tournament is depicted in Figure~\ref{fig:C4C4}.

\begin{figure}
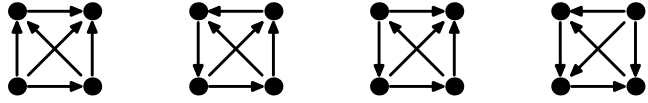

\begin{center}
\epsfbox{tourn-ck-4.mps}
\hskip 1cm
\epsfbox{tourn-ck-2.mps}
\hskip 1cm
\epsfbox{tourn-ck-5.mps}
\hskip 1cm
\epsfbox{tourn-ck-3.mps}
\end{center}
\caption{The tournaments $T^4$, $C^4$, $L^4$ and $W^4$.}
\label{fig:CLTW}
\end{figure}

\begin{figure}
\begin{center}
\epsfbox{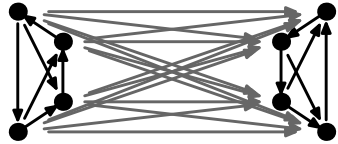}
\end{center}
\caption{The $8$-vertex tournament with the tournament matrix $\frac{J_8+D''_8}{2}$.
         The edges between the two copies of $C^4$ are drawn in gray to better display the structure of the tournament.}
\label{fig:C4C4}
\end{figure}

\begin{theorem}
\label{thm:c48}
Let $W$ be a tournamenton.
It holds that $C(W,4)=c(4)=4/3$ if and only if $W$ is weakly isomorphic to the carousel tournamenton $W_C$, and
it holds that $C(W,8)=c(8)=332/315$ if and only if $W$ is weakly isomorphic to the carousel tournamenton $W_C$.
\end{theorem}

\begin{proof}
If $W$ is weakly isomorphic to the the carousel tournamenton $W_C$, then $C(W,4)=C(W_C,4)=c(4)$ and $C(W,8)=C(W_C,8)=c(8)$.
Hence, we focus on proving the converse implications and start with the one concerning cycles of length four.
Let $W$ be a tournamenton such that $C(W,4)=c(4)$.
Let $T_n$, $n\in\NN$, be an $n$-vertex $W$-random tournament, and
let $B_n$ be the skew-symmetric matrix such that $\frac{J_n+B_n}{2}$ is the tournament matrix of $T_n$.
Note that the limit of $C(T_n,4)$ is $C(W,4)$ with probability one.

Similarly to the proof of Theorem~\ref{thm:c8}, we obtain using Lemma~\ref{lm:midterms} that
\[C(T_n,4)\le 1+\frac{1}{n^4}\Trace D_n^4-\frac{4}{n^3}||B_nj||^2+O(n^{-1}),\]
where $j$ is the vector with all entries equal to one.
It follows that
\[C(W,4)=\lim_{n\to\infty} C(T_n,4)\le 1+\lim_{n\to\infty}\frac{1}{n^4}\Trace D_n^4=c(4).\]
Furthermore, equality can hold only if
\begin{equation}
\lim_{n\to\infty}\frac{\|B_nj\|^2}{n^3}=0\label{eq:Bnj}
\end{equation}
and the proportion of principal submatrices of $B_n$ of order four that are sign-equivalent to $D_4$ tends to $1$.
The latter implies that $d(T^4,W)+d(C^4,W)=1$ and $d(L^4,W)+d(W^4,W)=0$,
i.e., the only $4$-vertex tournaments with positive density in $W$ are $T^4$ and $C^4$.
We conclude that all $4$-vertex subtournaments of $T_n$ are $T^4$ and $C^4$ (with probability one),
in particular, the in-neighborhood and the out-neighborhood of every vertex of $T_n$ is transitive.
If \eqref{eq:Bnj} holds, then $1/2\in\wsigma(W)$ and the tournamenton $W$ is regular by Proposition~\ref{prop:reg}.
Hence, the in-degree of every vertex of $T_n$ is close to $n/2$ with high probability,
which implies that $W$ is a limit of the carousel tournaments described at the end of Section~\ref{sec:prelim}.
Since the tournamenton $W_C$ is also a limit of the carousel tournaments,
the tournamentons $W$ and $W_C$ are weakly isomorphic.

We now deal with the case of cycles of length eight.
Let $W$ be a tournamenton such that $C(W,8)=c(8)$.
As in the previous case, we conclude that $W$ is regular and
the only $8$-vertex tournaments with positive density in $W$ are those
whose tournament matrix $A$ satisfies that $2A-J_n$ is sign-equivalent to $D_8$ and $D'_8$.
A~straightforward case analysis yields that every skew-symmetric matrix $B$ of order nine such that
every principal submatrix of order eight of $B$ is sign-equivalent to $D_8$ or $D'_8$
satisfies that $B$ is sign-equivalent to $D_9$.
It follows that the only $9$-vertex tournaments with positive density in $W$ are those
whose tournament matrix $A$ satisfies that $2A-J_n$ is sign-equivalent to $D_9$.
Consequently, the only $4$-vertex tournaments with positive density in $W$ are those
whose tournament matrix $A$ satisfies that $2A-J_n$ is sign-equivalent to $D_4$,
i.e., the tournaments $T^4$ and $C^4$.
Analogously to the previous case, we now conclude that the tournamentons $W$ and $W_C$ are weakly isomorphic.
\end{proof}

\bibliographystyle{bibstyle}
\bibliography{tourn-ck}

\end{document}